\documentclass[10pt, a4paper, twoside]{amsart}
\usepackage{amsthm}
\usepackage{amsmath}
\usepackage{amssymb}
\usepackage{amscd}
\usepackage{mathtools}
\usepackage[all]{xy}
\usepackage{indentfirst}
\usepackage{comment}
\usepackage{microtype}
\usepackage{tikz}
\usepackage{tikz-cd}
\usetikzlibrary{patterns}
\usepackage{enumitem}

\newtheorem{thm}{\bf Theorem}[section]
\newtheorem{prop}[thm]{\bf Proposition}
\newtheorem{lem}[thm]{\bf Lemma}

\newtheorem{cor}[thm]{\bf Corollary}

\newtheorem*{thm*}{\bf Theorem}
\newtheorem*{cor*}{\bf Corollary}

\theoremstyle{definition}

\newtheorem{rem}[thm]{\it Remark}

\newtheorem*{df*}{\bf Definition}
\newtheorem*{not*}{\bf Notation}
\newtheorem*{dfs*}{\bf Definitions}
\newtheorem*{ack*}{\bf Acknowledgements}
\newtheorem*{dfrem*}{\bf Definition and Remark}
\newtheorem*{conv*}{\bf Convention}

\def\P{\mathbb{P}}
\def\C{\mathbb{C}}

\def\Q{\mathbb{Q}}

\def\R{\mathbb{R}}

\def\O{\mathcal{O}}

\DeclareMathOperator{\Aut}{Aut}
\DeclareMathOperator{\Bir}{Bir}

\DeclareMathOperator{\Spec}{Spec}
\DeclareMathOperator{\Proj}{Proj}

\DeclareMathOperator{\Pic}{Pic}

\DeclareMathOperator{\Top}{top}
\DeclareMathOperator{\Nef}{Nef}
\DeclareMathOperator{\Eff}{Eff}

\DeclareMathOperator{\NS}{NS}

\title[\tiny Higher-dimensional Calabi-Yau varieties with dense sets of rational points]
{Higher-dimensional Calabi-Yau varieties with dense sets of rational points}
\date{\today}
\author{Fumiaki Suzuki}
\address{UCLA Mathematics Department, Box 951555, Los Angeles, CA, 90095-1555}
\email{suzuki@math.ucla.edu}

\begin{document}

\begin{abstract}
We construct higher-dimensional Calabi-Yau varieties defined over a given number field with Zariski dense sets of rational points.
We give two elementary constructions in arbitrary dimensions as well as another construction in dimension three which involves certain Calabi-Yau threefolds containing an Enriques surface.
The constructions also show that potential density holds for 
(sufficiently) 
general members of the families.
\end{abstract}
\maketitle

\section{Introduction}
For a smooth projective variety $X$ defined over a number field, one can ask whether the set of rational points is Zariski dense in $X$.
It is expected that the set of rational points reflects the positivity of the canonical bundle $\omega_{X}$ (cf. the Bombieri-Lang conjecture) and it is tempting to study the intermediate case $\omega_{X}=\O_{X}$.
In this case, $X$ belongs to the class of special varieties introduced by Campana \cite{C} which is conjecturally the same as that of varieties where rational points are potentially dense, that is, Zariski dense after passing to some finite field extension.
An interesting subcase is given by abelian varieties, for which potential density is well-known.

%It is challenging to consider the subcase given by 
%It is challenging to consider the case of
Another important and yet challenging subcase is given by
Calabi-Yau varieties in a {\it strict sense},
i.e., smooth projective varieties $X$ with $\omega_{X}=\O_{X}$ and $H^{0}(X,\Omega_{X}^{i})=0$ for all $0<i<\dim X$, which are simply connected over $\C$.
For elliptic K3 surfaces and K3 surfaces with infinite automorphism groups, potential density holds due to Bogomolov-Tschinkel \cite{BT3}.
Moreover, there are several works on K3 surfaces over the rational numbers with Zariski dense sets of rational points; for instance, quartic K3 surfaces have long been studied \cite{E, LMvL, SD1, SD2}.

Very little is known in higher dimensions.
It is stated by Tschinkel \cite[after Problem 3.5]{T} that it would be worthwhile to find non-trivial examples of Calabi-Yau threefolds over number fields with 
Zariski 
dense sets of rational points.
It is only recent that the first such examples were actually obtained:
Bogomolov-Halle-Pazuki-Tanimoto \cite{BHPT} studied Calabi-Yau threefolds with abelian surface fibrations
and showed potential density for threefolds (including simply connected ones) constructed by Gross-Popescu \cite{GP}.
However, it is not immediately clear whether their method can be used to determine the minimal field extensions over which rational points become Zariski dense.

In this short note, we construct higher-dimensional Calabi-Yau varieties in a strict sense defined over a given number field with Zariski dense sets of rational points.
We give two elementary constructions in arbitrary dimensions (Section \ref{C1}, \ref{C2}) as well as another construction in dimension three which involves certain Calabi-Yau threefolds containing an Enriques surface (Section \ref{C3}).
The constructions also show that potential density holds for 
(sufficiently) 
general members of the families.

The third construction is a by-product of the author's attempt to analyze in detail the recent construction due to Ottem-Suzuki \cite{OS} of a pencil of Enriques surfaces with non-algebraic integral Hodge classes.
For the third construction, our example contains no abelian surface and is a unique minimal model in its birational equivalence class,
 thus a theorem of Bogomolov-Halle-Pazuki-Tanimoto \cite[Theorem 1.2]{BHPT} cannot be applied.
For all the constructions, elliptic fibration structures are crucial.

We work over a number field unless otherwise specified.

\begin{ack*}
The author wishes to thank Lawrence Ein, 
Yohsuke Matsuzawa, 
John Christian Ottem,
Ramin Takloo-Bighash, and Burt Totaro for interesting discussions.
He also thanks the anonymous referee for useful suggestions.
\end{ack*}

\section{Construction I}\label{C1}
The first idea is to construct elliptic Calabi-Yau varieties in a strict sense (i.e., $H^{0}(X,\Omega^{i}_{X})=0$ for all $0<i<\dim X$ and simply connected over $\C$) whose base spaces are rational and which admits infinitely many sections.
For that purpose, we introduce variants of Schoen's Calabi-Yau fiber products of rational elliptic surfaces \cite{Sch}.

We let $S\subset \P^{1}\times \P^{2}$ be a smooth hypersurface of bi-degree $(1,3)$.
Then the first projection defines an elliptic fibration $f\colon S\rightarrow \P^{1}$.
Moreover, via the second projection, $S$ is the blow-up of $\P^{2}$ along the nine points given by the intersection of two cubic curves, hence rational.
For $n\geq 3$, we let $Y\subset \P^{1}\times \P^{n-1}$ be a smooth hypersurface of bi-degree $(1,n)$.
The first projection restricts to a fibration into Calabi-Yau hypersurfaces $g\colon Y\rightarrow \P^{1}$ and $Y$ is again rational via the second projection.
We assume that over any point in $\P^{1}$ either $f$ or $g$ is smooth, which is satisfied if $Y$ is general with respect to $S$.
We define $X=S\times_{\P^{1}}Y$ and let $\pi\colon X\rightarrow \P^{1}$ be the natural projection.

\begin{lem}\label{lCYI}
The fiber product $X$ is a Calabi-Yau $n$-fold in a strict sense.
\end{lem}
\begin{rem}
If $S_{\C}$ is very general, it is classical that the elliptic fibration $f\colon S_{\C}\rightarrow \P^{1}_{\C}$ admits infinitely many sections, 
which implies that the same holds for the natural projection $X_{\C}\rightarrow Y_{\C}$.
This provides examples of Calabi-Yau varieties in a strict sense defined over $\C$ containing infinitely many rational divisors in all dimensions $\geq 3$.
We will construct below such an example over a given number field.
\end{rem}
\begin{proof}[Proof of Lemma \ref{lCYI}]
It is immediate to see that $X$ is smooth.
The fiber product $X$ is a complete intersection in $\P^{1}\times \P^{2}\times \P^{n-1}$ of a hypersurface of tri-degree $(1,3,0)$ and that of tri-degree $(1,0,n)$.
We have $\omega_{X}=\O_{X}$ by the adjunction formula and an easy computation shows that $H^{i}(X,\O_{X})=0$ for all $0<i<n$.

For simple connectedness over $\C$,
Schoen proved this result for $n=3$ (see \cite[Lemma 1.1]{Sch0} for the strategy).
In fact, his method works for $n\geq 3$ and the argument goes as follows.
Let $U\subset \P^{1}_{\C}$ be the open subset over which $\pi\colon X_{\C}\rightarrow \P^{1}_{\C}$ is smooth and let $V=\pi^{-1}(U)$.
The natural map $\pi|_{V}\colon V\rightarrow U$ is topologically locally trivial and we let $F$ be a fiber.
We note that $\pi\colon X_{\C}\rightarrow \P^{1}_{\C}$ admits a section since $f$ and $g$ do, so does $\pi|_{V}\colon V\rightarrow U$.
Then we have a commutative diagram
\[
\xymatrix{
\pi_{1}(F)\ar[r] &\pi_{1}(V)\ar@{->>}[r]\ar@{->>}[d] &\pi_{1}(U)\ar@{->>}[d]\ar@/_1pc/[l]\\
 & \pi_{1}(X_{\C})\ar@{->>}[r]& \pi_{1}(\P^{1}_{\C}) \ar@/_1pc/[l],
}
\]
where the upper row is exact by the homotopy long exact sequence. 
Chasing the diagram and using the fact that $\pi_{1}(\P^{1}_{\C})$ is trivial, we are reduced to showing that $\pi_{1}(F)$ has the trivial image in $\pi_{1}(X_{\C})$.
Writing $F=F_{1}\times F_{2}$, where $F_{1}$ is a fiber of $f$ and $F_{2}$ is a fiber of $g$,
we see that
\[\pi_{1}(F)=\pi_{1}(F_{1})\times \pi_{1}(F_{2}).
\]
Now it is enough to verify that the image of $\pi_{1}(F_{1})$ and $\pi_{1}(F_{2})$ in $\pi_{1}(X_{\C})$ are trivial.
This is immediate since $\pi_{1}(F_{1})\rightarrow \pi_{1}(X_{\C})$ (resp. $\pi_{1}(F_{2})\rightarrow \pi_{1}(X_{\C})$) factors through the fundamental group of a section of $X_{\C}\rightarrow S_{\C}$ (resp. $X_{\C}\rightarrow Y_{\C}$) and 
since $S_{\C}$ (resp. $Y_{\C}$) is simply connected (it is rational).
The proof is complete.
\end{proof}

We give a construction of $X$ defined over $\Q$ such that $X(\Q)$ is Zariski dense in $X$.
We start from constructing $S$ defined over $\Q$ such that the elliptic fibration $f\colon S\rightarrow \P^{1}$ admits infinitely many sections over $\Q$, or equivalently,
the generic fiber $E/\Q(t)$ admits a $\Q(t)$-rational point and the Mordel-Weil group $E(\Q(t))$ has a positive rank.
The construction is as follows.

Let $C\subset \P^{2}$ be an elliptic curve defined over $\Q$ with a Zariski dense set of $\Q$-rational points.
Let $O\in C(\Q)$ (resp. $P\in C(\Q)$) be the origin (resp. a non-torsion point).
Let $D\subset \P^{2}$ be another elliptic curve defined over $\Q$ which goes through both $O$ and $P$ and which intersects transversally with $C$.
Let $S\subset \P^{1}\times \P^{2}$ be the hypersurface of bi-degree $(1,3)$ defined over $\Q$ corresponding to the pencil of elliptic curves generated by $C$ and $D$.
A zero-section of $f\colon S\rightarrow \P^{1}$ is given by the $(-1)$-curve over $O$ 
and the Mordel-Weil group $E(\Q(t))$ has a positive rank since the image of the specialization homomorphism $E(\Q(t))\rightarrow C(\Q)$ has a positive rank.
We conclude that $S$ has the desired property.

Let $Y$ be smooth, defined over $\Q$, and general so that over any point in $\P^{1}$ either $f$ or $g$ is smooth.
Then $X=S\times_{\P^{1}}Y$ is a Calabi-Yau $n$-fold in a strict sense defined over $\Q$.
Moreover the elliptic fibration $X=S\times_{\P^{1}}Y\rightarrow Y$ admits infinitely many sections over $\Q$ by construction.
We note that the set $Y(\Q)$ is Zariski dense in $Y$ since $Y$ is rational over $\Q$.
The following theorem is now immediate:

\begin{thm}\label{t1}
The set $X(\Q)$ is Zariski dense in $X$.
\end{thm}

\section{Construction II}\label{C2}
We let $X\subset (\P^{1})^{n+1}$ be a smooth hypersurface of multi-degree $(2,\cdots, 2)$.
The following is an immediate consequence of the Lefschetz hyperplane section theorem:

\begin{lem}
If $n\geq 2$, then $X$ is a Calabi-Yau $n$-fold in a strict sense.
\end{lem}

We give a construction of $X$ defined over $\Q$ such that $X(\Q)$ is Zariski dense in $X$.
The construction relies on the following theorem due to Bogomolov-Tschinkel.
We recall that a multi-section of an elliptic fibration is {\it saliently ramified} if it is ramified in a point which lies in a smooth elliptic fiber.
%The following theorem is due to Bogomolov-Tschinkel:

\begin{thm}[\cite{BT1,BT2}]\label{BT}
Let $\phi\colon \mathcal{E} \rightarrow B$ be an elliptic fibration over a number field $K$.
If there exists a saliently ramified multi-section $\mathcal{M}$ such that $\mathcal{M}(K)$ is Zariski dense in $\mathcal{M}$,
then $\mathcal{E}(K)$ is Zariski dense in $\mathcal{E}$.
\end{thm}
\begin{proof}
We sketch the proof for the convenience of the reader.
We may assume that $\phi\colon \mathcal{E}\rightarrow B$ is smooth by shrinking $B$,
and let $\phi_{\mathcal{J}}\colon \mathcal{J}=\Pic^{0}_{\mathcal{E}/B}\rightarrow B$ be the corresponding Jacobian fibration
(see \cite[Chapter 8, 9]{BLR} for the basic properties),
which is an abelian scheme.
We consider a family of line bundles of degree $0$
\[d\Delta_{\mathcal{E}}-\mathcal{E}\times_{B}\mathcal{M}\in \Pic(\mathcal{E}\times_{B}\mathcal{E})/\Pic(\mathcal{E})\subset \Pic_{\mathcal{E}/B}(\mathcal{E}),\]
where $d=\deg(\mathcal{M}/B)$, and let
\[
\tau\colon \mathcal{E}\rightarrow \mathcal{J}, \, p\mapsto dp - \mathcal{M}_{\phi(p)}
\]
be the induced morphism.
It is direct to see that $\tau$ is \'etale finite.
Then $\tau(\mathcal{M})$ cannot be contained in the $m$-torsion part $\mathcal{J}[m]$ for any positive integer $m$,
otherwise $\phi|_{\mathcal{M}}\colon \mathcal{M}\rightarrow B$ would be smooth, which contradicts the assumption.
Now Merel's theorem (or Mazur's theorem when $K=\Q$) implies 
that there exists a non-empty Zariski open subset $U\subset B$ such that 
for any $b\in \phi(\mathcal{M}(K))\cap U$ 
there exists a rational point $p_{b}\in \mathcal{M}(K)\cap\mathcal{E}_{b}$
such that $\tau(p_{b})$ is non-torsion on the fiber $\mathcal{J}_{b}$.
Finally, the fiberwise action of the Jacobian fibration on $\mathcal{E}$ translates rational points on $\mathcal{M}$, which concludes the proof.
\end{proof}

We start from constructing a smooth hypersurface $X_{1}\subset \P^{1}\times \P^{1}$ of bi-degree $(2,2)$ defined over $\Q$ such that $X_{1}(\Q)$ is Zariski dense in $X_{1}$.
We set
\[
\P^{1}\times \P^{1}=\Proj \Q[S_{1},T_{1}]\times \Proj \Q[S_{2},T_{2}].
\]
For instance, we can take $X_{1}$ to be the hypersurface defined by the equation
\[
%S_{1}^2S_{2}T_{2}+S_{1}T_{1}(2S_{2}^2+2S_{2}T_{2}+3T_{2}^{2})+T_{1}^2T_{2}(S_{2}+T_{2})=0.
S_{1}^{2}S_{2}T_{2}+S_{1}T_{1}(S_{2}^{2}+2S_{2}T_{2}+2T_{2}^{2})+T_{1}^{2}T_{2}(2S_{2}+2T_{2})=0.
\] 
Then $X_{1}$ is an elliptic curve with a non-torsion point defined over $\Q$, hence with a Zariski dense set of $\Q$-rational points, which can be directly seen as follows.
It is immediate to see that $X_{1}$ is smooth.
Moreover, if we let
\[
O=([1:0],[1:0]), P=([4:-1],[1:1])  \in X_{1}(\Q),
\]
then $P-O\in \Pic^{0}(X_{1})$ is of infinite order.
Indeed,
setting
\[
\P^{2}=\Proj\Q[U, V, W],
\]
the rational map $\P^{1}\times \P^{1}\dashrightarrow \P^{2}$ given by
\[
U=S_{1}T_{2}+T_{1}S_{2}, \, V=S_{1}T_{2}-T_{1}S_{2}, \, W=S_{1}T_{2}+T_{1}S_{2}+2T_{1}T_{2}
\]
transforms $X_{1}$ into its Weierstrass model
\[
V^{2}W=U^{3}-UW^{2}+W^{3},
\]
which is, in terms of inhomogeneous coordinates $u=U/W, v=V/W$,
\[
v^{2}=u^{3}-u+1.
\]
It is now straightforward to see that $O$ (resp. $P$) is sent to the point at infinity (resp. $(u,v)=(3,5)$),
where the latter defines a point of infinite order by a theorem of Lutz and Nagell \cite[VIII. Corollary 7.2]{S}
as the $v$-value $5$ is non-zero and $5^{2}$ does not divide $4\cdot (-1)^{3}+27\cdot 1^{2}=23$.
The claim follows.

For $n>1$, we set
\[
(\P^{1})^{n+1}=\Proj \Q[S_{1},T_{1}]\times \cdots \times \Proj \Q[S_{n+1},T_{n+1}]
\]
and inductively define $X_{n} \subset (\P^{1})^{n+1}$ to be a general hypersurface of multi-degree $(2,\cdots, 2)$ defined over $\Q$ containing $X_{n-1}$ as the fiber of the projection $pr_{n+1}\colon X_{n}\rightarrow \P^{1}$ over $T_{n+1}=0$.

\begin{lem}\label{l0}
The hypersurface $X_{n}$ is smooth and $X_{n-1}$ is a saliently ramified multi-section of the elliptic fibration 
%$pr_{1,\cdots, n-1}\colon X_{n}\rightarrow (\P^{1})^{n-1}$
 given by the projection 
$pr_{1,\cdots, n-1}\colon X_{n}\rightarrow (\P^{1})^{n-1}$ 
onto the first $n-1$ factors.
%copies of $\P^{1}$.
%: $pr_{1,\cdots, n-1}\colon X_{n}\rightarrow (\P^{1})^{n-1}$ .
\end{lem}
\begin{proof}
For smoothness, it is enough to show that $X_{n}$ is smooth around $X_{n-1}$.
This is obvious because $X_{n-1}$ is a fiber of the flat proper morphism $pr_{n+1}\colon X_{n}\rightarrow \P^{1}$ and $X_{n-1}$ is smooth by induction.
We are reduced to showing the second assertion.

Let $B\subset (\P^{1})^{n-1}$ be the branch locus, that is, the set of critical values of the composition $X_{n-1}\hookrightarrow X_{n}\xrightarrow{pr_{1,\cdots, n-1}} (\P^{1})^{n-1}$.
We note that this morphism is generically finite, but not finite when $n>3$.
We only need to prove that the fiber of $pr_{1,\cdots, n-1}\colon X_{n}\rightarrow (\P^{1})^{n-1}$ over a general point in $B$ is smooth.
Let $\Sigma\subset X_{n}$ be the set of critical points of $pr_{1,\cdots, n-1}\colon X_{n}\rightarrow (\P^{1})^{n-1}$.
By generality, it is sufficient to show that $\dim \Sigma\cap X_{n-1}=n-3$.
This can be checked by a direct computation as follows.
The equation of $X_{n}\subset (\P^{1})^{n+1}$ can be written as
\[
S_{n+1}^{2}F+S_{n+1}T_{n+1}G+T_{n+1}^{2}H=0
\]
for some multi-homogeneous polynomials $F, G, H$ in $S_{1}, T_{1},\cdots, S_{n}, T_{n}$ of multi-degree $(2,\cdots ,2)$,
where $F=0$ defines $X_{n-1}\subset (\P^{1})^{n}$.
If we moreover write $F$ as
\[F=S_{n}^{2}F_{1}+S_{n}T_{n}F_{2}+T_{n}^{2}F_{3}
\]
for some multi-homogeneous polynomials $F_{1},F_{2},F_{3}$ in $S_{1},T_{1},\cdots, S_{n-1}, T_{n-1}$ of multi-degree $(2,\cdots, 2)$,
then the set $\Sigma\cap X_{n-1}$ is defined by
\[
T_{n+1}=2S_{n}F_{1}+T_{n}F_{2}=S_{n}F_{2}+2T_{n}F_{3}=G=0.
\]
The first three equations define the ramification locus, that is, the set of critical points of the composition $X_{n-1}\hookrightarrow X_{n}\xrightarrow{pr_{1,\cdots, n-1}} (\P^{1})^{n-1}$,
which is of dimension $n-2$,
thus the four equations together define a closed subset of dimension $n-3$ by generality, as we wanted.
The proof is complete.
\end{proof}

Now Theorem \ref{BT} implies:

\begin{thm}\label{t2}
The set $X_{n}(\Q)$ is Zariski dense in $X_{n}$ for any $n\geq 1$.
\end{thm}

%We prove potential density for any members of the family of Calabi-Yau $n$-folds introduced in this section.
%We also prove a result on potential density.

%\begin{thm}
%Let $X \subset (\P^{1})^{n+1}$ be a smooth hypersurface of multi-degree $(2^{n+1})$ over a number field.
%Then rational points are potentially dense on $X$.
%\end{thm}
%\begin{proof}
%The proof is by induction on $n$.

%When $n=1$, the statement follows from the well-known fact that potential density holds for elliptic curves.

%When $n>0$, if a smooth fiber of some projection $pr_{i}\colon X\rightarrow \P^{1}$ to $\P^{1}$ is a saliently ramified multi-section for some elliptic fibration $pr_{j_{1},\cdots, j_{n-1}}\colon X\rightarrow (\P^{1})^{n-1}$,
%then the statement follows by induction and from Theorem \ref{BT}.
%Otherwise, all the projections $pr_{i_{1},\cdots, i_{k}}\colon X\rightarrow (\P^{1})^{k}$ are isotrivial.
%Using the fine moduli of elliptic curves, one can conclude that $X$ is rationally dominated by the product of elliptic curves, which again implies the statement.

%The proof is complete.
%\end{proof}

\begin{rem}
For a smooth hypersurface $X$ in $(\P^{1})^{n+1}$ of multi-degree $(2,\cdots, 2)$, 
the birational automorphism group $\Bir(X_{\C})$ is infinite by Cantat-Oguiso \cite[Remark 3.4]{CO}.
It would be possible to give a proof of the density of rational points by using the action of $\Bir(X_{\C})$.
\end{rem}

\section{Calabi-Yau threefolds containing an Enriques surface}\label{CY3}
In this section, we work over the complex numbers.
We construct a Calabi-Yau threefold containing an Enriques surface and prove basic properties of the threefold, which will be used in Section \ref{C3}.

Let
%$\mathcal{P}=\P_{\P^{2}}(\mathcal{O}^{\oplus 3}\oplus \mathcal{O}(1))$
$\Pi=\P_{\P^{2}}(\mathcal{O}^{\oplus 3}\oplus \mathcal{O}(1))$. 
%$\P=\P_{\P^{2}}(\mathcal{O}^{\oplus 3}\oplus \mathcal{O}(1))$.
On the projective bundle 
$\Pi$,
%$\P$, 
we consider a map of vector bundles
\[
u\colon \O^{\oplus 3}\rightarrow \O(2H_{1})\oplus \O(2H),
\]
where $H_{1}$ (resp. $H$) is the pull-back of the hyperplane section class on $\P^{2}$ (resp. the tautological class on $\Pi$).
%$\P$). 
Let $X$ be the rank one degeneracy locus of $u$.
%We assume that $X$ is smooth and the intersection $X\cap \P_{\P^{2}}(\O^{\oplus 3})$ is transversal.

\begin{lem}\label{l1}
If $u$ is general, $X$ is a Calabi-Yau threefold.
We have the topological Euler characteristic $\chi_{\Top}(X)=c_{3}(T_{X})=-84$ and the Hodge numbers $h^{1,1}(X)=2$ and $h^{1,2}(X)=44$.
\end{lem}
\begin{proof}
Since the vector bundle $\O(2H_{1})\oplus \O(2H)$ is globally generated, $X$ is a smooth threefold by the Bertini theorem for degeneracy loci.
%The resolution of the ideal sheaf $I_{X}$ of $X$ has the form
%\[
%0\rightarrow \O(-4H_{1}-2H)\oplus \O(-2H_{1}-4H) \rightarrow \O(-2H_{1}-2H)^{\oplus 3}\rightarrow I_{X}\rightarrow 0.
%\]
%It is immediate to see that $H^{1}(X,\O_{X})=H^{2}(X,\O_{X})=0$.
Another projective model of $X$ is defined as the zero set of a naturally defined section of $\O(1)^{\oplus 3}$ on the projective bundle 
$\widetilde{\Pi}=\P_{\Pi}(\O(2H_{1})\oplus \O(2H))$.
%$\widetilde{\P}=\P_{\P}(\O(2H_{1})\oplus \O(2H))$.
Let $\widetilde{H}$ be the tautological class on $\widetilde{\Pi}$.
%$\widetilde{\P}$.
The adjunction formula gives $\omega_{X}=\O_{X}(\widetilde{H}-2H)$.
On the other hand, $\widetilde{H}-2H$ is the class of the intersection 
$X\cap \P_{\Pi}(\O(2H_{1}))$, 
%$X\cap \P_{\P}(\O(2H_{1}))$, 
which is empty, thus we have $\O_{X}(\widetilde{H}-2H)=\O_{X}$.
It follows that $\omega_{X}=\O_{X}$.
The rest of the statement is a consequence of a direct computation using the conormal exact sequence and the Koszul resolution of the ideal sheaf of $X$ in $\widetilde{\Pi}$.
 %$\widetilde{\P}$.
The proof is complete.
\end{proof}

We assume that $u$ is general in what follows.

\begin{lem}\label{l2}
The threefold $X$ admits an elliptic fibration $\phi \colon X\rightarrow \P^{2}$.
Moreover, $X$ contains an Enriques surface $S$ and the linear system $|2S|$ defines a K3 fibration $\psi\colon X\rightarrow \P^{1}$.
\end{lem}
\begin{rem}
There are several examples of Calabi-Yau threefolds containing an Enriques surface.
For instance, see Borisov-Nuer \cite{BN}.
\end{rem}
\begin{proof}[Proof of Lemma \ref{l2}]
A natural projection 
$\Pi\rightarrow \P^{2}$
%$\P\rightarrow \P^{2}$
restricts to a surjection $\phi\colon X\rightarrow \P^{2}$ with the geometric generic fiber a complete intersection of two quadrics in $\P^{3}$, which is an elliptic curve.
The morphism $\phi$ has equidimensional fibers, hence it is flat.

Moreover, the map $u$ restricts to a map of vector bundles
\[
v\colon \O^{\oplus 3}\rightarrow \O(2,0)\oplus \O(0,2)
\]
on $\P_{\P^{2}}(\mathcal{O}^{\oplus 3})=\P^{2}\times \P^{2}$.
The intersection $X\cap \P_{\P^{2}}(\mathcal{O}^{\oplus 3})$ is the rank one degeneracy locus of $v$, which is an Enriques surface $S$
 by generality 
 (see \cite[Lemma 2.1]{OS}).
Then the linear system $|2S|$ defines a K3 fibration $\psi\colon X\rightarrow \P^{1}$ by \cite[Proposition 8.1]{BN}.
The proof is complete.
\end{proof}

\begin{cor}
The threefold $X$ contains no abelian surface.
\end{cor}
\begin{proof}
If $X$ contains an abelian surface $A$, then it is easy to see that the linear system $|A|$ defines an abelian surface fibration $\eta \colon X\rightarrow \P^{1}$.
Then pulling back $\NS(\P^{1}\times \P^{1})_{\R}$ by $(\psi, \eta)\colon X\rightarrow \P^{1}\times \P^{1}$ defines a two-dimensional linear subspace of $\NS(X)_{\R}$ which does not contain an ample divisor.
Therefore we should have $\rho(X)\geq 3$.
This is impossible since we have $\rho(X)=h^{1,1}(X)=2$ by Lemma \ref{l1}.
\end{proof}

\begin{lem}\label{l3}
$\Nef(X)=\overline{\Eff(X)}=\R_{\geq 0}[H_{1}]\oplus \R_{\geq 0}[S]$.
\end{lem}
\begin{proof}
Since the linear systems $|H_{1}|$ and $|2S|$ define the fibrations $\phi\colon X\rightarrow \P^{2}$ and $\psi\colon X\rightarrow \P^{1}$ respectively,
the divisors $H_{1}$ and $S$ are semi-ample but not big, so their classes give extremal rays in $\overline{\Eff(X)}$.
This finishes the proof.
\end{proof}

\begin{cor}
The threefold $X$ is a unique minimal model in its birational equivalence class.
Moreover, $\Bir(X)=\Aut(X)$ and these groups are finite.
\end{cor}
\begin{proof}
Let $f\colon X\dashrightarrow X'$ be a birational map with $X'$ a minimal model.
Then $f$ can be decomposed into a sequence of flops by a result of Kawamata \cite[Theorem 1]{Ka2}.
Note that any flopping contraction of a Calabi-Yau variety is given by a codimension one face of the nef cone (see \cite[Theorem 5.7]{Ka1}).
Since the codimension one faces $\R_{\geq 0}[H_{1}]$ and $\R_{\geq 0}[S]$ of $\Nef(X)$ give fibrations, it follows that $f$ is in fact an isomorphism.
For the last statement, we note that Oguiso proved that the automorphism group of any odd-dimensional Calabi-Yau variety in a wider sense with $\rho=2$ is finite.
The proof is complete.
\end{proof}

\begin{prop}\label{p}
The threefold $X$ is simply connected.
\end{prop}
\begin{proof}
Applying \cite[Lemma 5.2.2]{K} to the K3 fibration $\psi\colon X\rightarrow \P^{1}$, whose smooth fibers are simply connected, we are reduced to showing that $2S$ is the only one multiple fiber of $\psi$.
The K3 fibration $\psi\colon X\rightarrow \P^{1}$ and the morphism $X\rightarrow \P^{5}$ given by the linear system $|H|$ induce a morphism $X\rightarrow \P^{1}\times \P^{5}$,
which is the blow-up of a non-normal complete intersection $X_{0}$ in $\P^{1}\times \P^{5}$ of three hypersurfaces of bi-degree $(1,2)$ along the non-normal locus with the exceptional divisor $S$.
Now we only need to prove that the first projection $pr_{1}\colon X_{0}\rightarrow \P^{1}$ admits no multiple fibers.
This follows from the Lefschetz hyperplane section theorem.
The proof is complete.
\end{proof}

\section{Construction III}\label{C3}
We recall a system of affine equations for an Enriques surface introduced by Colliot-Th\'el\`ene--Skorobogatov--Swinnerton-Dyer \cite{CTSSD}.

\begin{prop}[\cite{CTSSD}, Proposition 4.1, Example 4.1.2; \cite{L}, Proposition 1.1]\label{CTSSD}
Let $k$ be a field of characteristic zero.
Let $c, d, f\in k[t]$ 
be polynomials of degree $2$ such that $c\cdot d\cdot (c-d)\cdot f$ is separable.
Let $S^{0}$ be the affine variety in $\mathbb{A}^{4}=\Spec k[t, u_{1}, u_{2}, u_{3}]$ 
defined by
\[
u_{1}^{2}-c(t) = f(t)u_{2}^{2}, \,u_{1}^{2}-d(t) = f(t)u_{3}^{2}.
\]
Then a minimal smooth projective model $S$ of $S^{0}$ is an Enriques surface.
The projection $S^{0} \rightarrow \mathbb{A}^{1}=\Spec k[t]$ induces an elliptic fibration $S\rightarrow \P^{1}$ with reduced discriminant $c\cdot d \cdot (c-d)\cdot f$
which admits double fibers over $f=0$ whose reductions are smooth elliptic curves.
\end{prop}

Now we apply Proposition \ref{CTSSD} to 
$k=\Q$ and 
\[c=-3t^{2}-2t+8,\, d=\frac{t^{2}-15t+16}{2},\, f=\frac{t^{2}+1}{2}.\]
Let $S$ be the corresponding Enriques surface.
We prove:

\begin{lem}\label{l6}
The surface $S$ has a Zariski dense set of $\Q$-rational points.
\end{lem}
\begin{proof}
We follow the strategy of the proof of \cite[Proposition 5]{Sko}.
It will be easier to work on the K3 double cover $\widetilde{S}$.
By setting 
\[w^{2}=\frac{t^{2}+1}{2},\, v_{2}=wu_{2},\, v_{3}=wu_{3},\]
we obtain the defining equations of its affine model $\widetilde{S}^{0}$ in 
$\mathbb{A}^{5}=\Spec \Q[t,u_{1},v_{2},v_{3},w]$:
\[
u_{1}^{2}-(-3t^{2}-2t+8) = v_{2}^{2}, \, u_{1}^{2}-\frac{t^{2}-15t+16}{2} = v_{3}^{2},\, \frac{t^{2}+1}{2}=w^{2}.
\]
The projection 
$\widetilde{S}^{0}\rightarrow C^{0}=\Spec\Q[t,w]/(\frac{t^{2}+1}{2}-w^{2})$ 
defines an elliptic fibration $\widetilde{S}\rightarrow C$ with reduced discriminant $c\cdot d\cdot (c-d)$.

We consider the curve $E^{0}\subset \widetilde{S}^{0}$ cut out by
\[
u_{1}=t-3,\,v_{2}=2t-1,
\]
which is isomorphic to the affine curve in 
$\mathbb{A}^{3}=\Spec \Q[t,v_{3}, w]$ 
defined by
\[
\frac{1}{2}(t+1)(t+2)=v_{3}^{2},\, \frac{t^{2}+1}{2}=w^{2}.
\]
Then $E^{0}$ gives an elliptic curve $E$.
We prove that $E(\Q)$ is Zariski dense in $E$.
Let $O\in E$ be the point given by
\[
t=-1,\, v_{3}=0, \, w=-1
\]
and $P\in E$ be the point given by
\[
t=7,\, v_{3}=-6,\, w=5.
\]
We only need to prove that $P-O\in \Pic^{0}(E)$ is of infinite order.
It is a simple matter to check that the map
\[
(t,v_{3},w)\mapsto \left(\frac{2(3t-4w-1)}{t-2w-1},\frac{8v_{3}}{t-2w-1}\right)
\]
gives a transformation into the Weierstrass model
\[
y^{2}=x^{3}-52x+144
\]
and sends $O$ (resp. $P$) to the point at infinity (resp. $(x,y)=(0,12)$).
Thus it is enough to verify that $(x,y)=(0,12)$ defines a point of infinite order.
This follows from a theorem of Lutz and Nagell \cite[VIII. Corollary 7.2]{S}
since the $y$-value $12$ is non-zero and $12^{2}$ does not divide $4\cdot (-52)^{3}+27\cdot 144^{2}$.

Moreover, $E$ is a saliently ramified multi-section of the elliptic fibration $\widetilde{S}\rightarrow C$.
Indeed, it is easy to verify that $E\rightarrow C$ is branched over 
\[t=-1,\, w=\pm 1,
\]
while $t=-1$ is not a root of the reduced discriminant 
\[
c\cdot d\cdot (c-d)= (-3t^{2}-2t+8)\left(\frac{t^{2}-15t+16}{2}\right)\left(\frac{-7t^{2}+11t}{2}\right).
\]

Now Theorem \ref{BT} shows that $\widetilde{S}(\Q)$ is Zariski dense in $\widetilde{S}$.
This in turn implies that $S(\Q)$ is Zariski dense in $S$.
The proof is complete.
\end{proof}

We define a compactification of $S^{0}$ in $\P^{2}\times \P^{2}$ as follows.
We set 
\[
\P^{2}\times \P^{2}=\Proj\Q[X_{0},X_{1},X_{2}]\times \Proj\Q[Y_{0},Y_{1},Y_{2}].
\]
On $\P^{2}\times \P^{2}$, we consider the map of vector bundles
\[
v\colon \O^{\oplus 3}\rightarrow \O(2,0)\oplus \O(0,2)
\]
given by the $2\times 3$ matrix
\[
\left(
\begin{array}{ccc}
X_{0}^{2}& X_{1}^{2}& X_{2}^{2}\\
2Y_{0}^{2}+6Y_{1}^{2}+4Y_{1}Y_{2}-16Y_{2}^{2} & 2Y_{0}^{2}-Y_{1}^{2}+15Y_{1}Y_{2}-16Y_{2}^{2} & Y_{1}^{2}+Y_{2}^{2}
\end{array}
\right).
\]
Let $S'$ be the rank one degeneracy locus of $v$.
It is straightforward to see that $S'$ is indeed a compactification of $S^{0}$.
The surface $S'$ is a local complete intersection, so in particular, Gorenstein.
The surface $S'$ has isolated singular points and blowing up along the points gives a crepant resolution $S\rightarrow S'$.

Finally, we give the construction of the Calabi-Yau theefold.
%a construction.
On the projective bundle 
$\Pi=\P_{\P^{2}}(\mathcal{O}^{\oplus 3}\oplus \mathcal{O}(1))$, 
%$\P=\P_{\P^{2}}(\mathcal{O}^{\oplus 3}\oplus \mathcal{O}(1))$, 
we let
\[
u\colon \mathcal{O}^{\oplus 3}\rightarrow \O(2H_{1})\oplus \O(2H)
\]
be general among all maps defined over $\Q$ which restrict to $v$ on $\P_{\P^{2}}(\mathcal{O}^{\oplus 3})=\P^{2}\times \P^{2}$.
We define $X$ to be the rank one degeneracy locus of $u$.

\begin{thm}\label{D}
The threefold $X$ is a Calabi-Yau threefold in a strict sense defined over $\Q$ with a Zariski dense set of $\Q$-rational points.
Moreover, $X$ satisfies the following geometric properties:
\begin{enumerate}
\item $X_{\C}$ admits K3 and elliptic fibrations;
\item $X_{\C}$ contains no abelian surface;
\item $X_{\C}$ is a unique minimal model in its birational equivalence class;
\item $\Bir(X_{\C})=\Aut(X_{\C})$ and these groups are finite.
\end{enumerate}
\end{thm}
\begin{proof}
One can check that $X$ is smooth and the proofs in Section \ref{CY3} still go through.
It remains to show that $X(\Q)$ is Zariski dense in $X$.
By a similar argument to that in Lemma \ref{l0}, $S'$ is a saliently ramified multi-section of the elliptic fibration $\phi \colon X\rightarrow \P^{2}$.
Since $S'(\Q)$ is Zariski dense in $S'$ by Lemma \ref{l6}, the result follows from Theorem \ref{BT}.
The proof is complete.
\end{proof}

One can be more explicit.
We fix a non-zero element 
$Z\in H^{0}(\Pi,\O_{\Pi}(H-H_{1}))$.
%$Z\in H^{0}(\P,\O_{\P}(H-H_{1}))$.
Let $u$ be given by the matrix
\[
\left(
\begin{array}{ccc}
P_{1}& Q_{1}& R_{1}\\
P_{2} & Q_{2}& R_{2}
\end{array}
\right),
\]
where
\[
P_{1}=X_{0}^{2},\, Q_{1}=X_{1}^{2},\, R_{1}=X_{2}^{2}
\]
and
\begin{align*}
P_{2}&=(X_{1}^{2}+X_{2}^{2})Z^{2}+(X_{1}Y_{1}+X_{2}Y_{2})Z+2Y_{0}^{2}+6Y_{1}^{2}+4Y_{1}Y_{2}-16Y_{2}^{2},\\
Q_{2}&=(X_{0}^{2}+X_{2}^{2})Z^{2}+(X_{0}Y_{0}+X_{2}Y_{2})Z+2Y_{0}^{2}-Y_{1}^{2}+15Y_{1}Y_{2}-16Y_{2}^{2},\\
R_{2}&=(X_{0}^{2}+X_{1}^{2})Z^{2}+(X_{0}Y_{0}+X_{1}Y_{1})Z+Y_{1}^{2}+Y_{2}^{2}.
\end{align*}
Then {\it Macaulay2} shows that the corresponding $X$ is smooth and that the discriminant locus $\Delta\subset \P^{2}$ of $\phi\colon X\rightarrow \P^{2}$ and the branch locus $B\subset \P^{2}$ of $\phi|_{S}\colon S\rightarrow \P^{2}$ meet properly.
Consequently, the set $X(\Q)$ is Zariski dense in $X$.

\end{document}